\documentclass{amsart}

\usepackage{amsmath}
\usepackage{amssymb}
\usepackage{graphicx}

\newtheorem{theorem}{Theorem}
\newtheorem*{theorem*}{Theorem}
\newtheorem{lemma}{Lemma}
\newtheorem*{lemma*}{Lemma}
\newtheorem{corollary}{Corollary}
\newtheorem{conjecture}{Conjecture}
\newcommand\be{\begin{equation}}
\newcommand\ee{\end{equation}}

\begin{document}

\title[Prime Difference Champions]{Prime Difference Champions}

\author[ S. Funkhouser]{S. Funkhouser}
\address{SPAWAR Systems Center Atlantic, One Innovation Drive, North Charleston, SC 29419}

\author[D. A. Goldston]{D. A. Goldston}
\address{Department of Mathematics and Statistics, San Jos\'{e} State University, One Washington Square, San Jos\'{e}, CA 95192-0103, USA}

\email{daniel.goldston@sjsu.edu}

\author[D. Sengupta] {D. Sengupta}
\address{Department of Mathematics and Computer Science, Elizabeth City State University, 1704 Weeksville Road, Elizabeth City, NC 27909}

\email{dcsengupta@ecsu.edu}
\author[J. Sengupta]{J. Sengupta}
\address{Department of Mathematics and Computer Science, Elizabeth City State University, 1704 Weeksville Road, Elizabeth City, NC 27909}

\email{jdsengupta@ecsu.edu}

\subjclass[2000]{Primary 11N05; Secondary 11P32, 11N36}

\keywords{Differences between primes; Hardy-Littlewood prime pair conjecture; Jumping champion; Maximal prime gaps; Primorial numbers; Sieve methods; Singular series}

\begin{abstract} 
A Prime Difference Champion (PDC) for primes up to $x$ is defined to be any element of the set of one or more differences that occur most frequently among all positive differences between primes $\le x$.  Assuming an appropriate form of the Hardy-Littlewood Prime Pair Conjecture we can prove that for sufficiently large $x$ the PDCs run through the primorials. Numerical results also provide evidence for this conjecture as well as other interesting phenomena associated with prime differences.  Unconditionally we prove that the PDCs go to infinity and further have asymptotically the same number of prime factors when counted logarithmically as the primorials.  
\end{abstract}

\maketitle

\thispagestyle{empty}

\section{Introduction} 
The study of prime gaps has been historically one of the most insightful probes of the nature of the primes.  
The gaps, referring here strictly to the positive differences between consecutive primes, are a particularly natural subject of investigation because they exhibit global characteristics that are amenable to rigorous analytical treatment.  
Some of the most important theorems related to primes are those concerning the gaps.  
Most prominently, we may state the Prime Number Theorem equivalently in terms of the asymptotic behavior of the average local gap.  

An elemental component of the study of prime gaps is the counting function, 
\be
N(x,d) 
 \mathrel{\mathop:}= \sum_{\substack{p_{n\!+\! 1}\leq x \\ p_{n\!+\!1}-p_n = d}}\! 1 \, ,
\ee
giving the number of gaps of size $d$ within the sequence of primes no greater than some $x$, 
where $p_n$ is the $n$-th prime.  
As an example Figure 1 is a plot of $N(10^5,d)$.  
\begin{figure}
\includegraphics[width=0.7\linewidth]{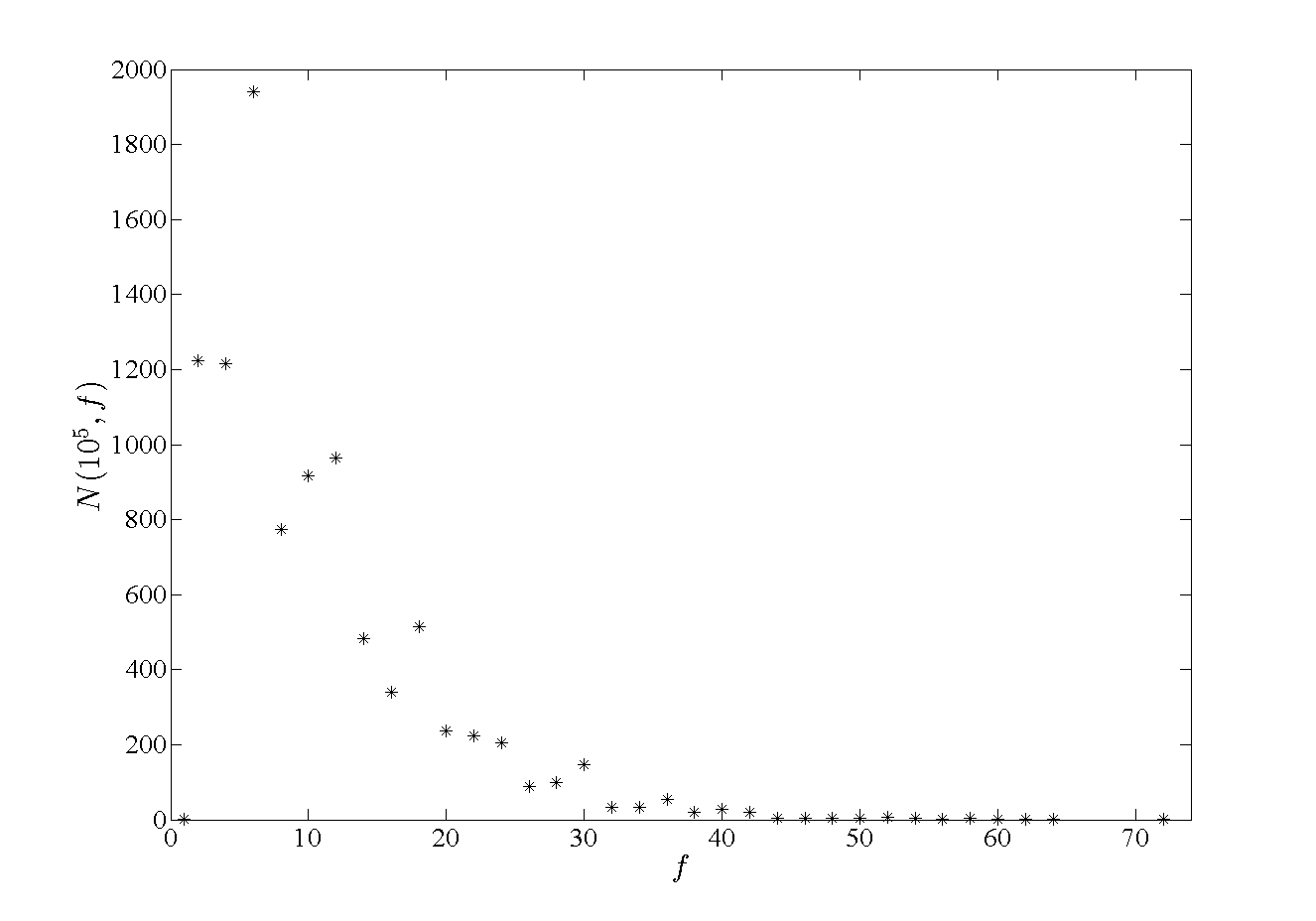}
\caption{ Gap count for $x=10^5$, showing only non-zero $N(x,d)$.
\label{Fig_N1E5} }
\end{figure} 
From $N(x,d)$ we may readily obtain the average gap among primes not exceeding $x$.  
The most commonly occuring gap (or gaps) among primes no greater than some $x$ is the value (or values) of $d$ for which $N(x,d)$ is maximal.  
With 
\be
N^{*}(x)
 \mathrel{\mathop:}= \max_{d} N(x,d) 
\ee
the set 
\be
J^{*}(x)
 \mathrel{\mathop:}= \{ d \colon N(x,d) = N^{*}(x)\} 
\ee 
formally describes the most common gap(s), known as the {\it prime jumping champion(s)} (PJC's),  for a given $x$.  
Figure \ref{Fig_PJC} shows the PJC's for all prime $x\le 1800$.  
\begin{figure}
\includegraphics[width=0.7\linewidth]{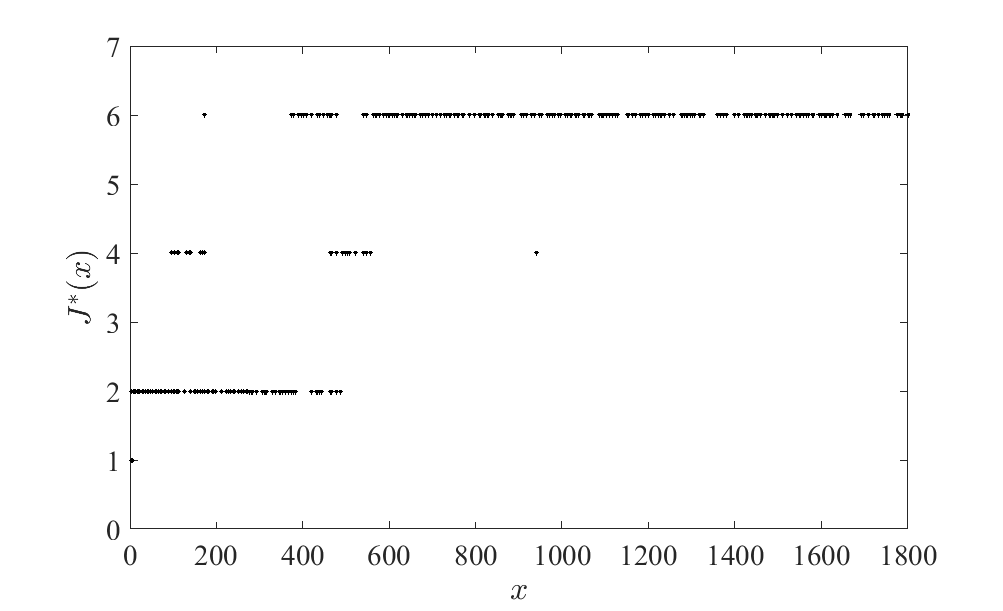}
\caption{ Prime jumping champions for all odd, prime $x\le1800$
\label{Fig_PJC} }
\end{figure}

Several analytical results concerning the PJC's are relevant here.  
In 1999, based on heuristic arguments and extensive numerical studies, Odlyzko, Rubinstein and Wolf (ORW) \cite{OdlyzkoRubinsteinWolf1999} conjectured that  the PJC's greater than unity are $4$ and the primorials, 
\be 
\left\{p_{k}^{\sharp}\right\}_{k=1}^\infty = 2, 6, 30, 210, 2310, \ldots \ , 
\ee
where 
\be  \label{eq5}
p_{k}^{\sharp}:=\prod_{j=1}^k {p_j} 
\ee
is the $k$-th primorial.  
They also advanced a weaker implication of this conjecture, namely that the PJC's are asymptotically infinite and that any given prime divides all sufficiently large PJC's.  
In 1980 Erdos and Straus \cite{ErdosStraus1980} proved, assuming the Hardy-Littlewood prime pair conjecture (HLPPC), that the PJC's tend to infinity.  
In 2011 Goldston and Ledoan extended the method described in \cite{ErdosStraus1980} to give a proof that any given prime will divide all sufficiently large jumping champions \cite{GoldstonLedoan2011}.  
Soon thereafter they gave also a proof that sufficiently large PJC's run through the primorials assuming a sufficiently strong form of the Hardy-Littlewood prime pair and prime triple conjectures \cite{GoldstonLedoan2015}.  

It is important to note that, despite the body of work devoted to the prime gaps, we know virtually nothing unconditionally about the PJC's. 
Numerically 6 is the PJC for $947\le x\le 10^{15}$, and it unlikely that numerical work will ever find any PJC other than 6 for ranges where computations are feasible.  
Unconditionally we can not even disprove that the PJC is 2 for all sufficiently large $x$, or eliminate any given even number as sometimes being a PJC.  
The study of differences between primes, however, need not be restricted to differences between consecutive primes.  
In fact, the HLPPC is for the generalized (positive) differences between primes not exceeding a given $x$, and it is natural to first study pair differences in a given sequence before examining more complicated gap questions.  

The purpose of this work is to present a broad investigation of the prime differences and their so-called champions.  
We find that, in contrast to the situation with the gaps, rich behaviors are evident in computationally accessible ranges and, more importantly, both conditional and unconditional theorems may be formulated to explain many prominent characteristics of the PDC's.  
The work is organized as follows.  
Section \ref{S_Num} contains definitions of the central terms and some basic numerical studies, in which significant regularities are immediately evident.  
In Section \ref{S_HLPPC} is found an overview of the HLPPC, with appropriate adaptations necessary for this present analysis.  
Section \ref{S_NumHL} contains the results of several numerical studies in support of the HLPPC.  
Section \ref{S_Sketch} outlines how the HLPPC may be used to probe the nature of the prime differences.  
In Section \ref{S_ProofOfTheorem1} it is proven under condition of the HLPPC that the PDC's run through the primorials for sufficiently large $x$.  
Section \ref{S_LogSums} reviews some relevant properties of logarithmic sums.  
In Section \ref{S_PDCsInfinity} it is proven unconditionally that the PDCs tend to infinity and have many prime factors.  

\section{Counting prime differences \label{S_Num}}
Let $p$ and $p'$ denote primes and let $d$ be a positive integer. 
Analogously to $N(x,d)$ we define the counting function 
\be
G(x,d)
 \mathrel{\mathop:}= \sum_{\substack{p,p'\leq x \\ p'-p = d}}\! 1
\ee
to give the number of prime differences equal to $d$ among primes no greater than $x$.  
Figure \ref{Fig_G1E5} shows the count of prime differences for $x=10^5$.  
For convenience we have excluded all odd differences, being those associated with the anomalously even $p_1=2$.  
Figure \ref{Fig_G1E5_zoom} shows a zoomed-in view of Figure \ref{Fig_G1E5}.  
\begin{figure}
\includegraphics[width=0.7\linewidth]{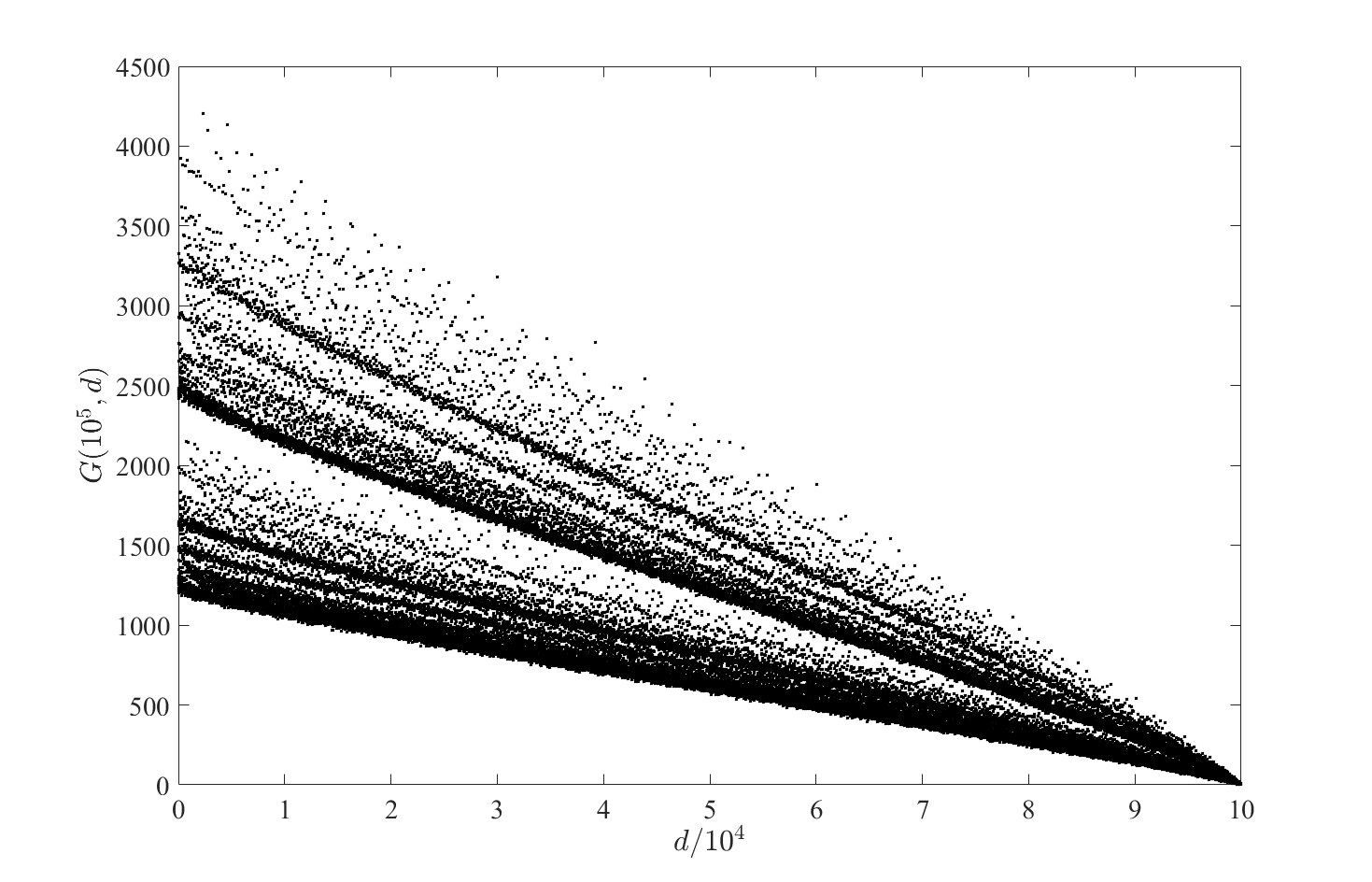}
\caption{ Difference count for $x=10^5$
\label{Fig_G1E5} }
\end{figure}
\begin{figure}
\includegraphics[width=0.7\linewidth]{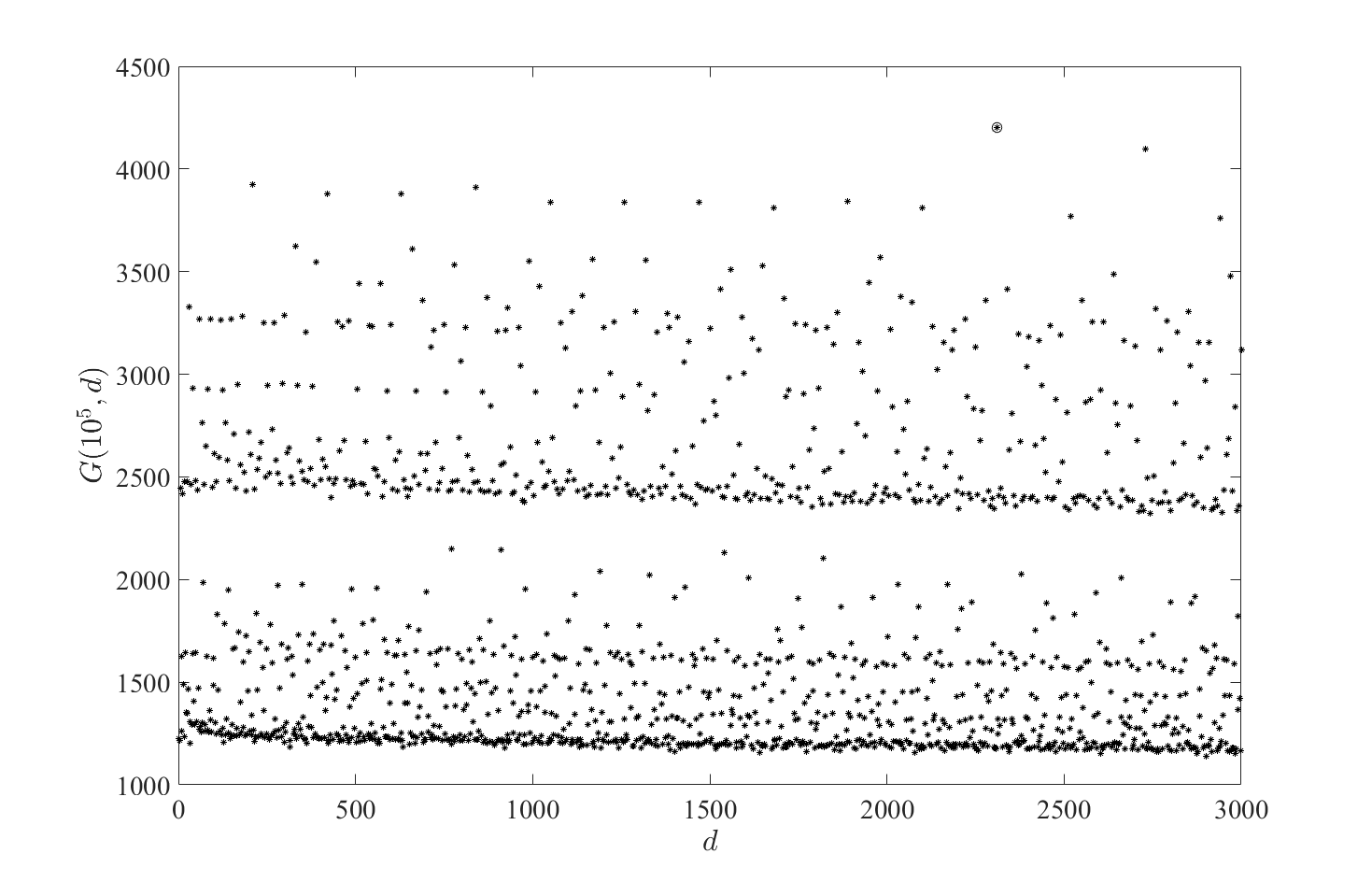}
\caption{Detail of difference count for $x=10^5$; the PDC is $2310$ and the corresponding count is circled.  
\label{Fig_G1E5_zoom} }
\end{figure}
Figures \ref{Fig_G1E5} and \ref{Fig_G1E5_zoom} are representative for all sufficiently large $x$ within the computationally explored range, which so far extends to $x=2\times 10^8$.  

A prime difference champion for a given $x$ is any $d$ for which $G(x,d)$ attains its maximum,  
\be
G^{*}(x)
 \mathrel{\mathop:}= \max_{d} G(x,d) \, .
\ee
Analogously to $J^*(x)$ we define formally the PDC (or PDC's) for a given $x$ by the set   
\be
D^{*}(x)
 \mathrel{\mathop:}= \{ d \colon G(x,d) = G^{*}(x)\} \, .
\ee
As $D^*(x)$ is a step function of $x$ with possible steps only when $x$ is a prime, we only need consider values of $x=p\ge 3$ in examining the behavior of the PDC's.    
For example we have $D^{*}(3) =\{ 1\}$, $D^{*}(5) =\{ 1, 2, 3\}$,  $D^{*}(7) =\{ 2\}$,  $D^{*}(11) =\{ 2,4\}$,  $D^{*}(13) =\{ 2\}$, $D^{*}(17) =\{ 2,4,6\}$. 
Figure \ref{Fig_PDC} is a plot of the PDCs for all prime $x\le 2\times 10^8$.  The dashed and dot-dashed lines in Figure (\ref{Fig_PDC}) are associated with Theorem 1 as introduced in Section \ref{S_Sketch}. 
\begin{figure}
\includegraphics[width=0.7\linewidth]{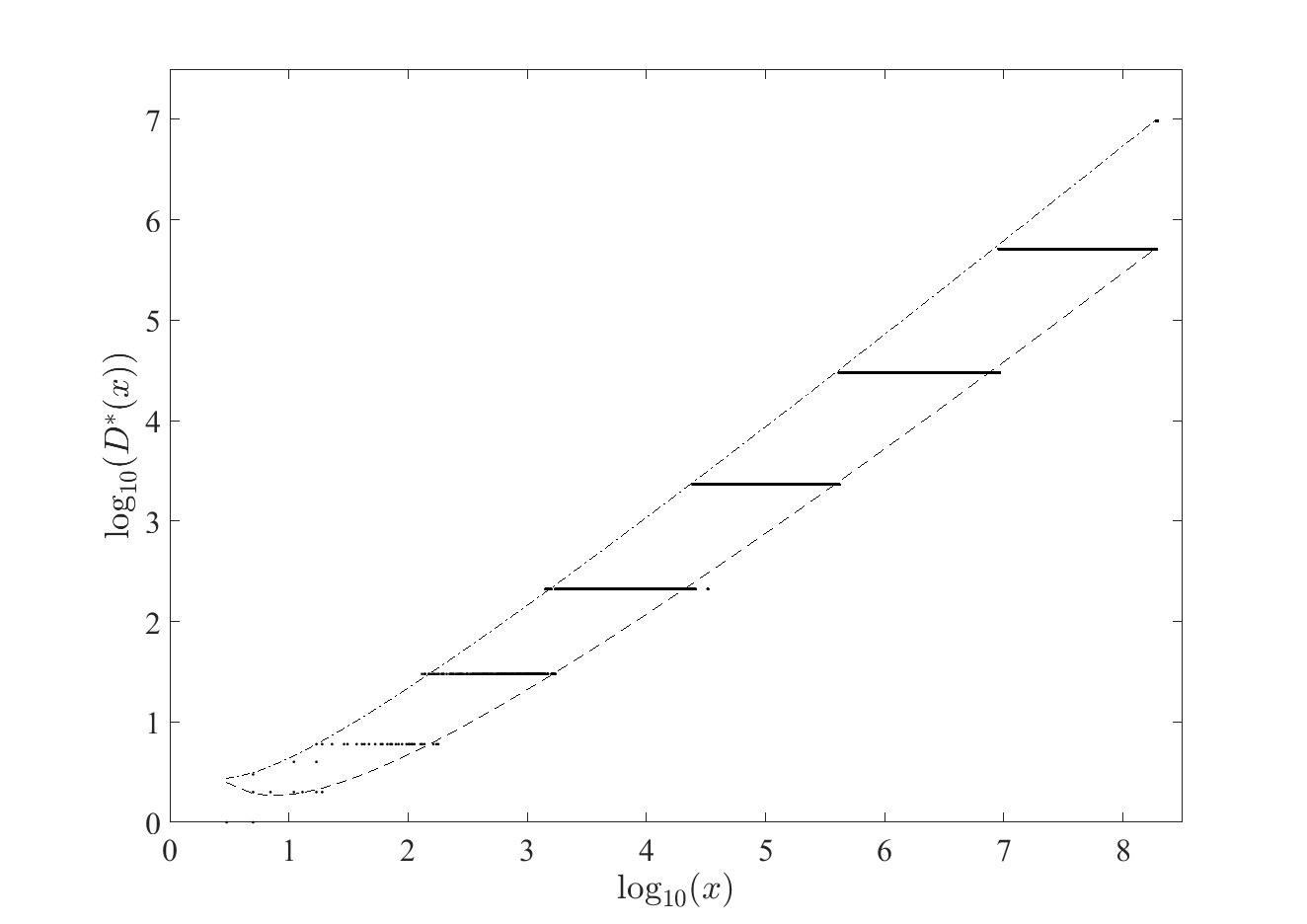}
\caption{ Prime difference champions for all odd, prime $x$ no greater than $2\times10^8$.  
The dot-dashed line is a plot of $\log_{10}( x/\log(x))$ and the dashed line is a plot of $\log_{10}(x/\log^2(x))$, in association with Theorem 1.  
\label{Fig_PDC} }
\end{figure}

The basic computational results exemplified in Figures \ref{Fig_G1E5} through \ref{Fig_PDC} intimate several new global properties of the prime differences and the PDC's.  
For sufficiently large $x$, $G(x,d)$ exhibits a rich structure bearing a statistically strong signature of periodicity \cite{FCSW}.  
As shown in Section \ref{S_HLPPC} we may attribute the periodic structure to behavior encapsulated in the HLPPC.  
Aside from $1$, $3$ and $4$ all PDC's in the numerically explored range of $x$ are primorials.  
Additionally the PDC's run through the primorials in a broadly step-wise manner as $x$ increases.  

It is instructive to review some of the details of the observed transitions in the PDC's.  
Table \ref{tablePDCtrans} shows the smallest $x$ and the largest known prime $x$ for which $D^*(x)$ includes each respective primorial within the observed range.  
Because the last observed instance of $9699690$ represents only the arbitrarily chosen end of calculations it does not represent any meaningful transition point and is ommitted from the table.  
Note that the step-wise transitions occur for $x$ in the vicinity of $p_n$ such that $n$ is roughly equal to the next largest primorial.  For example, the transition from $210$ to $2310$ begins at $n=2718$, and so on \cite{FCSW}.  
This empirical trend may be understood as a condition of the HLPPC, as formalized at the end of Section \ref{S_Sketch}.  
\begin{table}[ht] 
\caption{Transition points for PDC's}
\renewcommand\arraystretch{1.0}
\noindent\[
\begin{array}{|r|c|c|}
\hline
\text{primorial}	&	\text{smallest $x$ for}	&	\text{largest known prime $x$ for}	\\
					&	\text{occurrence as PDC}	&	\text{occurrence as PDC}		\\
\hline
6		&	p_7=17	&	p_{41}=179 \\
30		&	p_{32}=131	&	p_{269}=1723	\\
210		&	p_{224}=1423	 &	p_{3523}=32843  \\
2310	&	p_{2718}=24499		&	p_{35000}=414977  \\
30030	&	p_{34903}=413863	&	p_{609928}=9120277	\\
510510	&	p_{607867}=9087131	&	p_{10657197}=191945597	\\
9699690&	p_{10561154}=190107653		&	- \\
\hline
\end{array}
\]
\label{tablePDCtrans}
\end{table}

Within the transition regions the PDC is observed to exhibit some oscillatory behavior between the two primorials involved in the transition.  
For example, Fig. \ref{Fig_PDC_2310_30030} is a plot of $D^*(x)$ beginning with the first occurence of $30030$ and ending at the last occurence of $2310$.  
Note that during the transitions the PDC is observed only to be one or both of the two respective primorials involved.  
\begin{figure}
\includegraphics[width=0.7\linewidth]{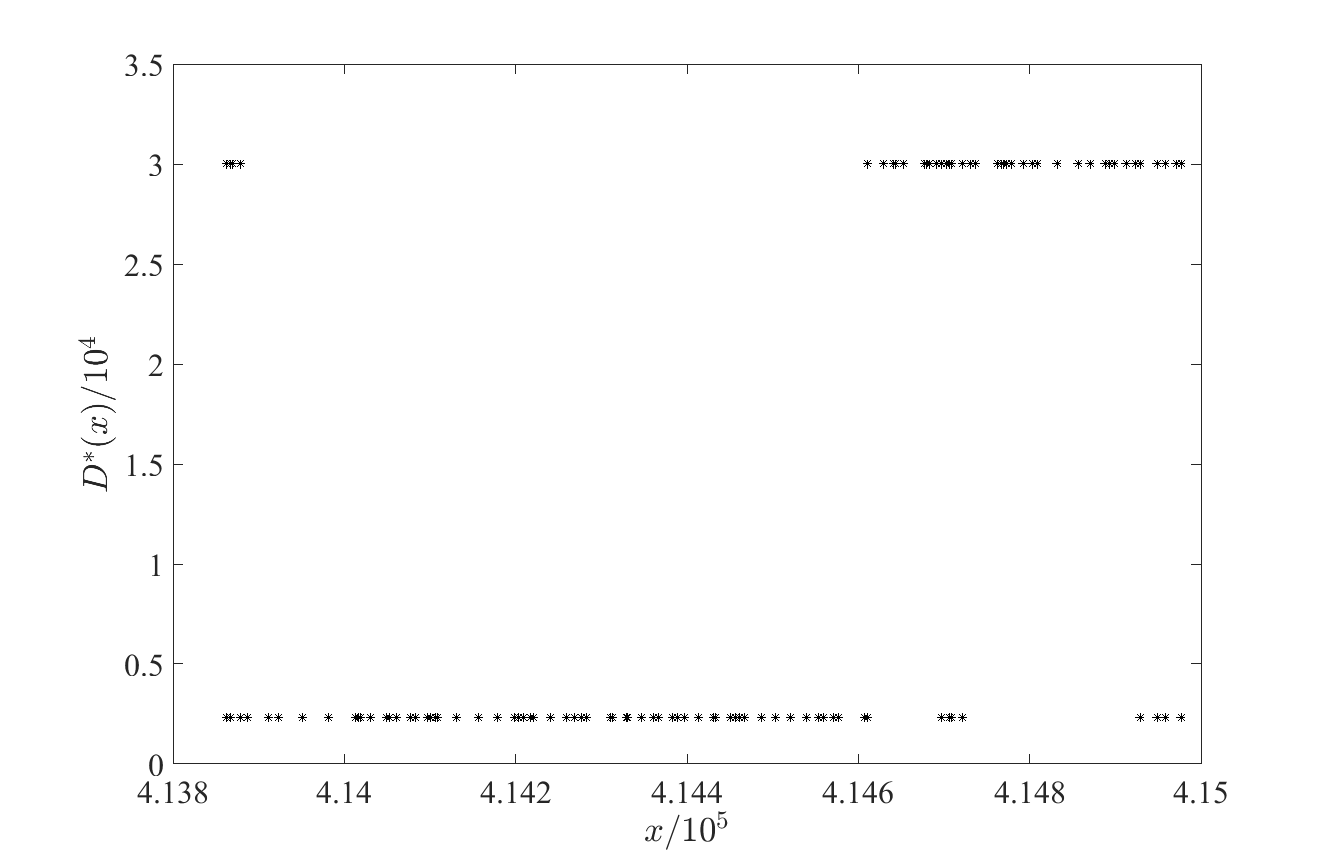}
\caption{ Zoom-in view of $D^*(x)$ over region where transition from $210$ to $2310$ occurs.   
\label{Fig_PDC_2310_30030} }
\end{figure}

\section{The Hardy-Littlewood prime pair conjecture \label{S_HLPPC}}
Hardy-Littlewood in their famous paper Partitio Numerioum III \cite{HardyLittlewood1923} conjectured an asymptotic formula for $G(x,d)$. This is often called the prime pair conjecture, and states that 
\be\label{HLpairs}
G(x,d) \sim\mathfrak{S}(d) \frac{x}{(\log x)^2}, \quad \mbox{as $x \to \infty$}.
\ee
The singular series $\mathfrak{S}(d)$ is defined for all integers $d \neq 0$ by
\be
\label{3.2} \mathfrak{S}(d) = \left\{ \begin{array}{ll}
      {\displaystyle 2C_2\prod_{\substack{p \mid d \\ p > 2}} \left(\frac{p - 1}{p - 2}\right),} & \mbox{if $d$ is  even, $d \neq 0$;} \\
      0,   & \mbox{if $d$ is odd;} \\
\end{array}
\right.
\ee
where
\be\label{3.3}
C_2
 = \prod_{p > 2}\left(1 - \frac{1}{(p - 1)^2}\right)
 = 0.66016\ldots
\ee
with the product extending over all primes $p > 2$.
We see that for $d$ a positive even integer the singular series may be written as
\be \label{singprod}
\mathfrak{S}(d) = 2C_2\prod_{\substack{p \mid d \\ p > 2}} \left(1 + \frac{1}{p - 2}\right),
\ee
and therefore $\mathfrak{S}(d)$ has local maximums when $d$ is a primorial. 

In the case of the prime number theorem we know that we obtain a better approximation by assuming that a prime $p$ has a density or probability of occuring of $\frac{1}{\log p}$ rather than taking the constant density $\frac{1}{\log x}$ for all the primes up to $x$.   We expect the same is true for prime pairs, and Hardy and Littlewood  conjectured that one should replace 
\be 
\frac{x}{(\log x)^2}  \quad \text{by}\quad   \text{li}_2(x) =  \int_2^x\frac{dt}{(\log t)^2} , 
\ee 
in which case the conjecture should hold with a much smaller error term. 

Although Hardy and Littlewood in this case did not specifically consider the situation where $d=d(x) \to \infty$, it is reasonable to suppose that a form of the HLPPC will hold in this situation.  We see in the definition of $G(x,d)$ that if $d> 0$ then the conditions $p, p'\le x$ and $p'= p+d$ implies that $p\le x-d$. Hence we see for $d> 0$ that 
\be 
G(x,d) =  \sum_{\substack{ p\le x-d \\ p+d \text{  is prime}}}1, 
\ee
and thus we may  conjecture for $d>0$ that since $p$ has density $\frac{1}{\log p}$ and $p+d$ has density $\frac{1}{\log (p+d)}$
\be\label{3.4}
G(x,d)  =\mathfrak{S}(d) \int_2^{x-d} \frac{dt}{\log t \log (t+d)} + E(x,d), \quad \mbox{as $x \to \infty$}, 
\ee 
where $E(x,d)$ represents an error term.  
For convenience let us define 
\be 
I(x,d) := \int_2^{x-d} \frac{dt}{\log t \, \log (t+d)}
\ee 
and
\be 
\widetilde{G}(x,d) := \mathfrak{S}(d) I(x,d) .
\ee
such that 
\be
G(x,d) = \widetilde{G}(x,d) + E(x,d) \, ,
\ee
where $\widetilde{G}(x,d)$ represents presumably the asymptotic difference count.  

A strong conjecture is that for $ 2\le d \le x - x^\epsilon$
\be  
E(x,d) \ll  (x-d)^{\frac12+\epsilon}.
\ee
The actual conjecture we need to resolve the present problems concerning the PDC's is much weaker than this, and is stated appropriately as follows.  
\begin{conjecture}  We have that \eqref{3.4} holds with  $E(x,d) = o( \frac{x}{(\log x)^4})$  uniformly for  $ 2\le d \le \frac89 x$.
\label{Cmain}
\end{conjecture}

\section{Numerical tests of the Hardy-Littlewood conjecture \label{S_NumHL}}
Before proceeding to the proofs it is worthwhile to examine some numerical results related to the behaviors expected of the HLPPC, as expressed in Conjecture \ref{Cmain}.  
The studies described here provide novel, detailed tests of the validity of the HLPPC.  

First let us examine the representative behaviors of the terms on the right-hand side of \eqref{3.4}.  
Figure \ref{Fig_Gana1E5} is a plot of $\widetilde{G}(10^5,d)$, for comparison to Figure \ref{Fig_G1E5}.  
Figure \ref{Fig_E1E5} shows the corresponding error term, $E(10^5,d)$.  
Finally Fig. \ref{Fig_I1E5} shows $I(10^5,d)$.  
The behaviors shown in these figures are representative for sufficiently large $x$, and the precision of the numerical integrations is such that the associated uncertainties are negligible on the scales of all of the figures.  

\begin{figure}
\includegraphics[width=0.7\linewidth]{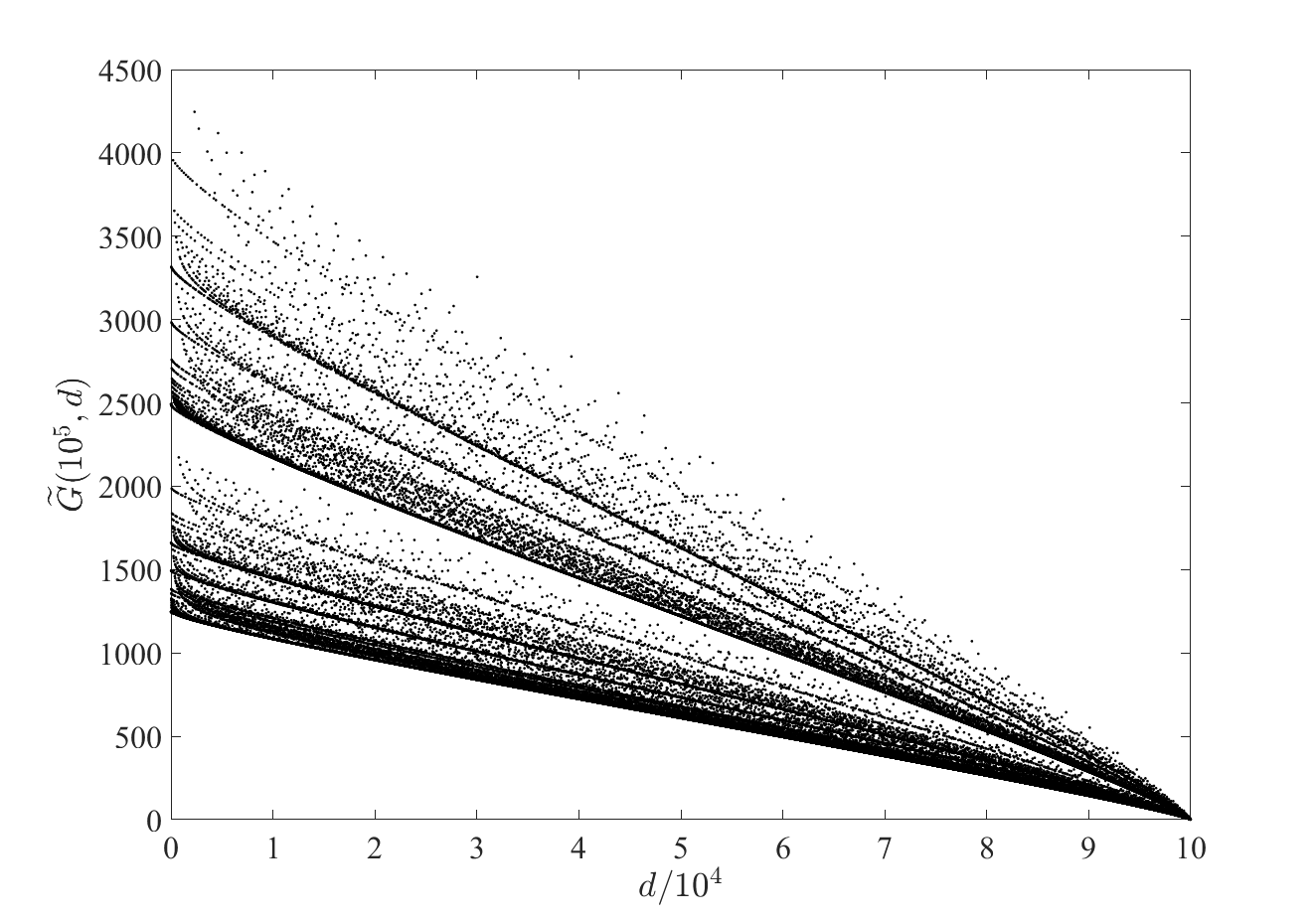}
\caption{Asymptotic difference count, $\widetilde{G}(x,d)$, for $x=10^5$.  Compare to Fig.\ref{Fig_G1E5}.  
\label{Fig_Gana1E5} }
\end{figure}

\begin{figure}
\includegraphics[width=0.7\linewidth]{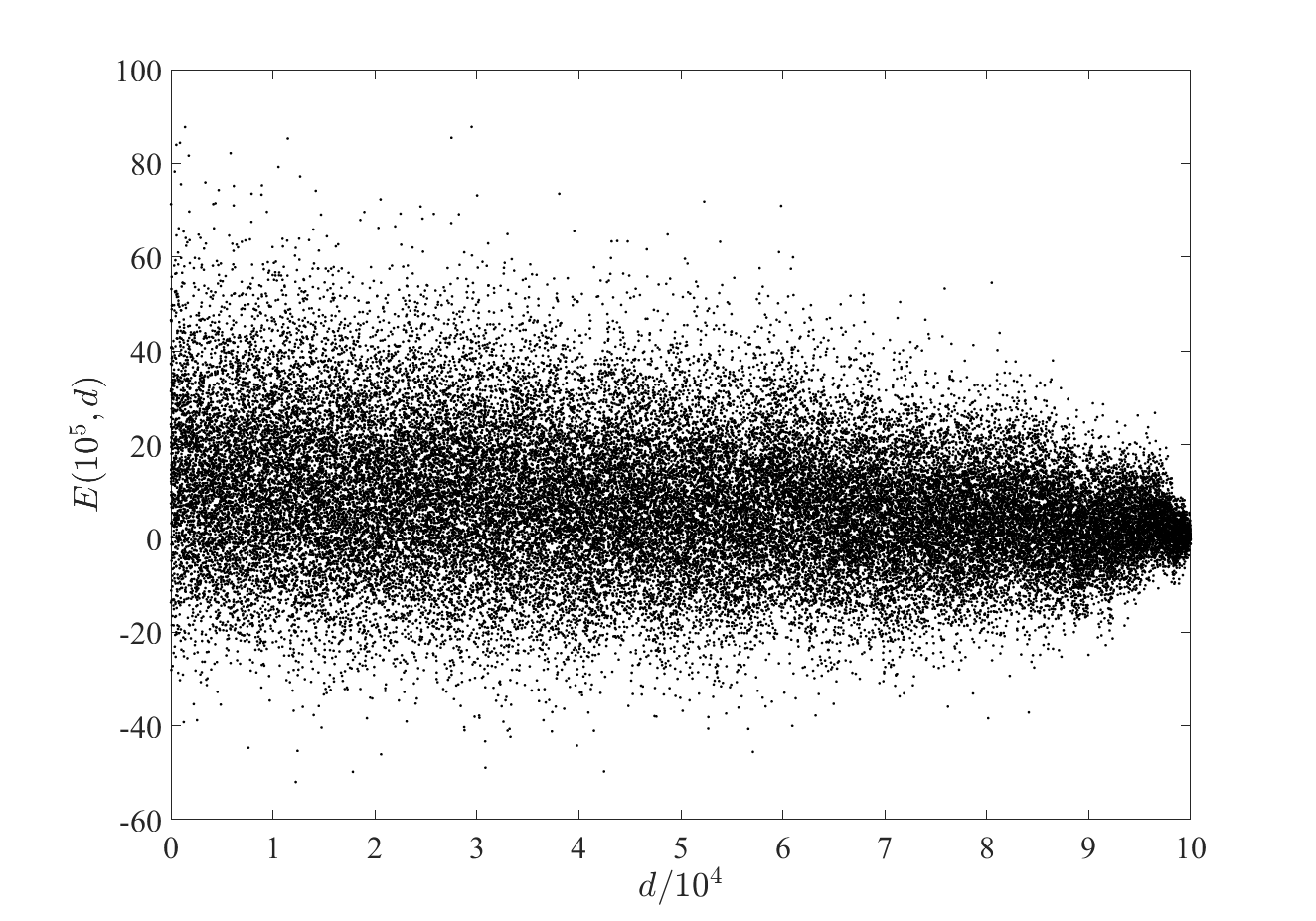}
\caption{ Error term associated with asymptotic count for $x=10^5$.  
\label{Fig_E1E5} }
\end{figure}

\begin{figure}
\includegraphics[width=0.7\linewidth]{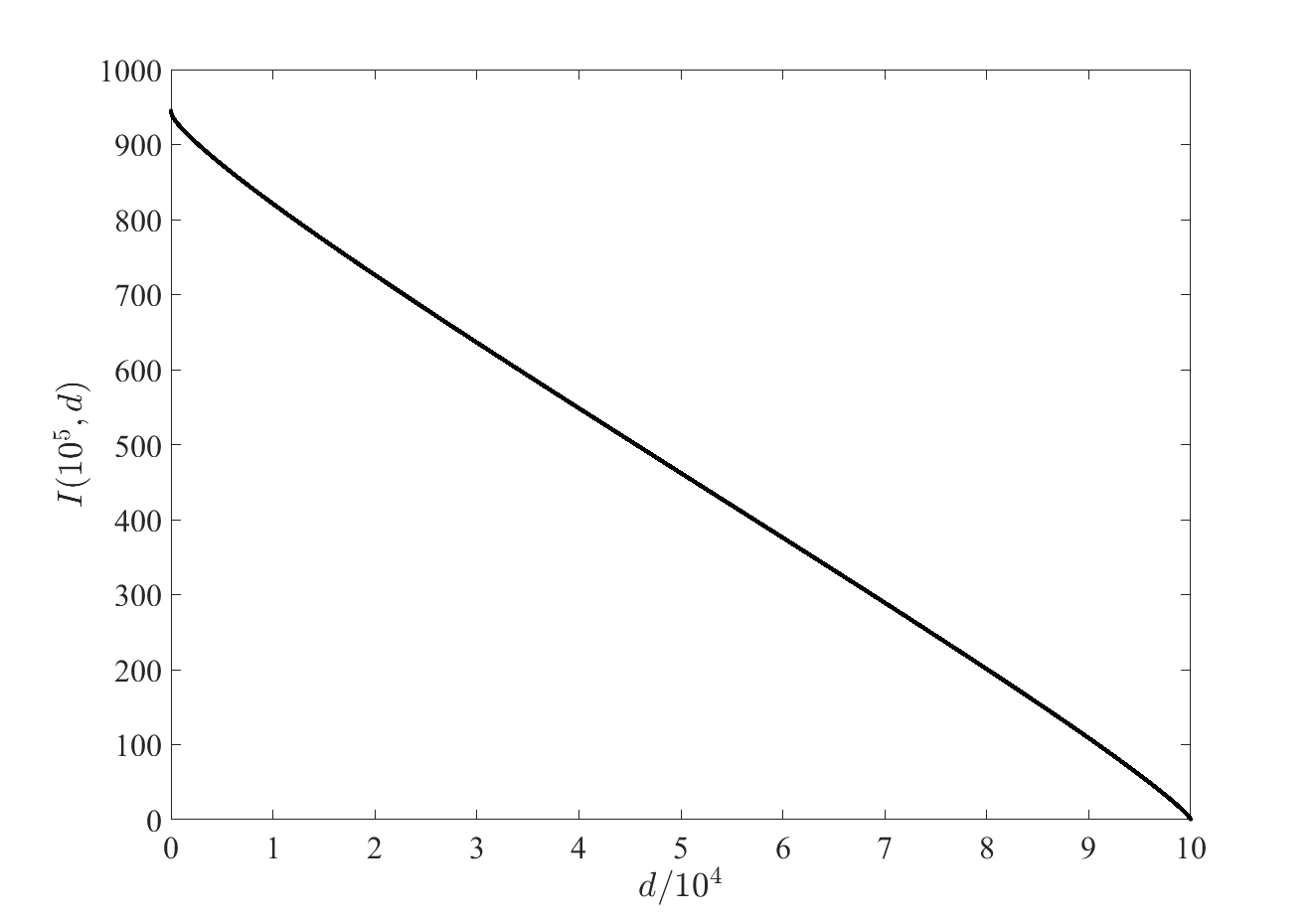}
\caption{Numerically computed integral term for $x=10^5$.  
\label{Fig_I1E5} }
\end{figure}

Next, note that we have 
\be
\sum_{d}{G(p_n,d)} = \frac{n(n-1)}{2}
\label{sumG}
\ee
by construction, where summations over $d$ are taken for $d=1,2,4,6,8,...$.  In accordance with the HLPPC we therefore expect 
\be
\sum_{d}{\widetilde{G}(p_n,d)} \sim \frac{n^2}{2} \, .
\label{Eq_sumGasym}
\ee
In order to test (\ref{Eq_sumGasym}) it is appropriate to define the relative error 
\be
\mu(x) := \frac{   \sum_{d}{  \left( \widetilde{G}(x,d) - G(x,d) \right) } }{ \sum_{d}{  G(x,d)  }}
\ee
Figure \ref{Fig_mu} is a plot of $\mu(x)$ for roughly logarithmically spaced, prime $x\in[10^4, 10^7]$.  
The dashed line is $1/\sqrt{\pi(x)}$, where $\pi(x)$ is the usual prime counting function, hence $\pi(p_n)=n$. 
An analysis of the possible relationship between $\mu(x)$ and a function of the form $c_{\mu}/\sqrt{\pi(x)}$, for some constant $c_{\mu}$, is reserved for future work.  
It is sufficient here to note that $\mu(x)$ would vanish asymptotically if the observed trends should persist {\it ad infinitum}.  
\begin{figure}
\includegraphics[width=0.7\linewidth]{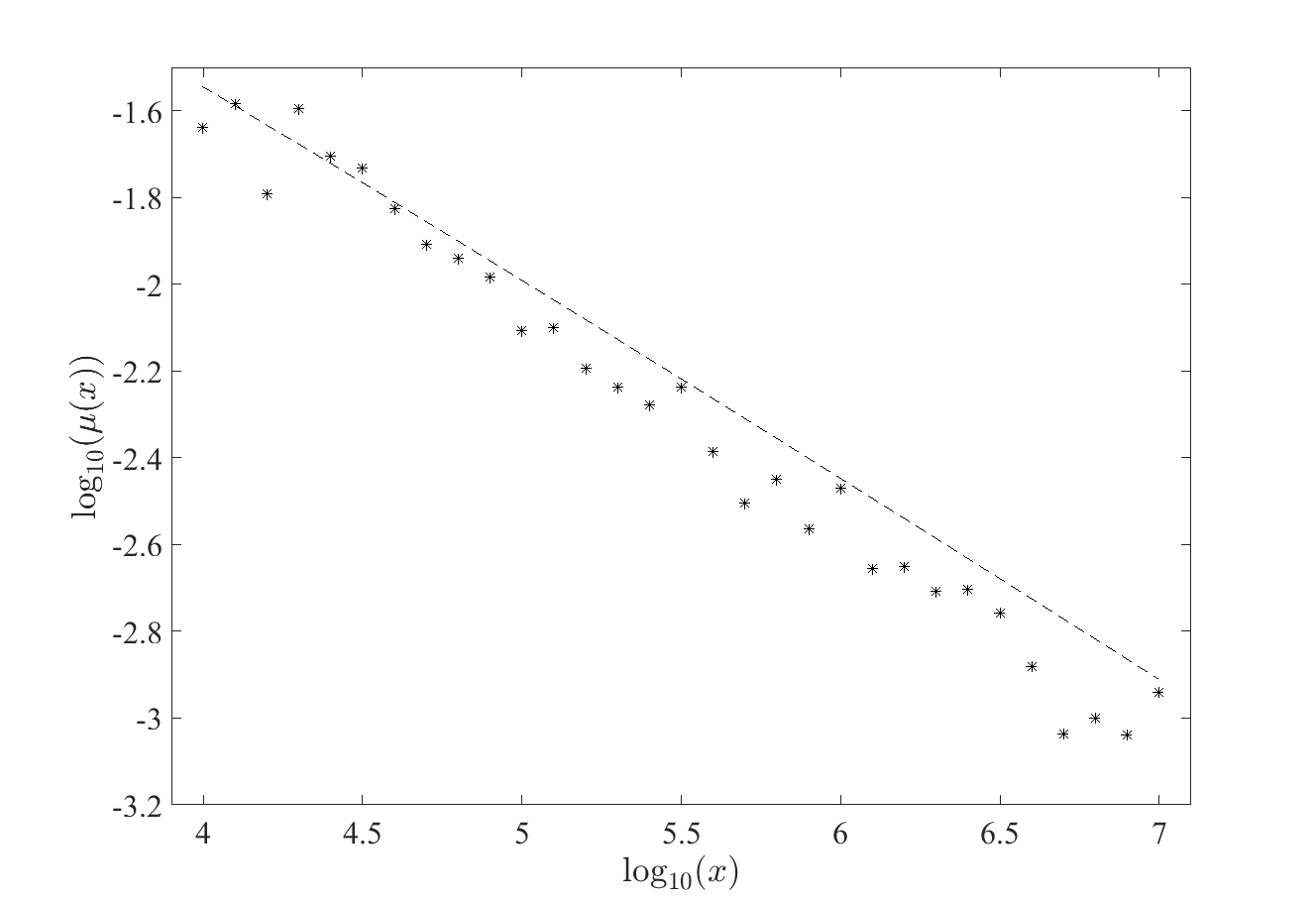}
\caption{ Relative error in the sum of $\widetilde{G}(x,d)$, for certain $x=p_n$, plotted with asterisks.  
The dashed line is a plot of $1/\sqrt{\pi(x)}$, in corresponding logarithmic fashion.
\label{Fig_mu} }
\end{figure}

It is also instructive to consider the variance, 
\be
\nu(x):= \sum_d{\left(  \widetilde{G}(x,d)-G(x,d)   \right)^2}   \, .
\ee
Figure \ref{Fig_nu} is a plot $\nu(x)/\pi(x)^2$ for the same particular $x=p_n$ as in Fig. \ref{Fig_mu}.  
Note that the plotted points are all in the vicinity of $0.16$.  
The asymptotic behavior of $\nu(x)$ and the apparent broad proportionality between $\nu(x)$ and $\pi(x)^2$ are reserved for future studies.  
\begin{figure}
\includegraphics[width=0.7\linewidth]{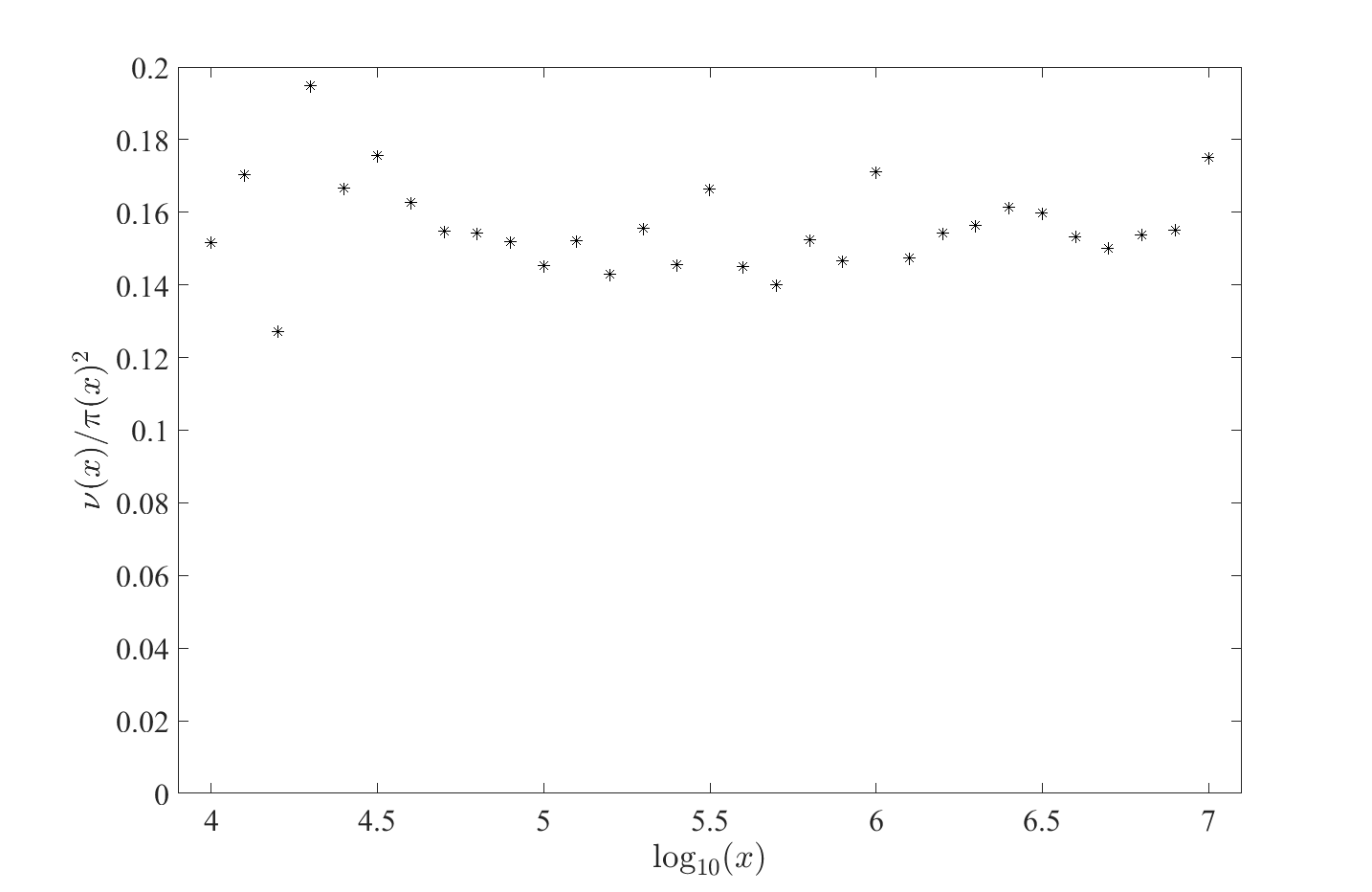}
\caption{Plot of variance, $\nu(x)$, divided by $\pi(x)^2$ for certain $x=p_n$.  
\label{Fig_nu} }
\end{figure}

\section{Sketch of Solution of the PDC Problem using Conjecture 1 \label{S_Sketch}}
To determine the PDCs, we use Conjecture 1 to obtain formulas for $G(x,d)$ in various ranges of $d$. 
We will assume here $2 \le d \le \frac89 x$. We want to evaluate $I(x,d)$ asymptotically, but we will never need to be more accurate than terms with size smaller than $o(\frac{x}{(\log x)^4})$. Hence, we take 
\be I(x,d) = \int_{\frac{x}{(\log x)^5}}^{x-d} \frac{dt}{\log t\, \log (t+d)}+ O(\frac{x}{(\log x)^5}),\ee
and use integration by parts to obtain
\be 
I(x,d) = \frac{x-d}{\log x\, \log (x-d)} +  \int_{\frac{x}{(\log x)^5}}^{x-d} \frac{(\log (t+d) + \frac{t}{t+d}\log t) dt}{(\log t \, \log (t+d))^2}+ O(\frac{x}{(\log x)^5}).
\ee
Since $ \frac{t}{t+d} = 1 -\frac{d}{t+d}$, we obtain 
\be 
I(x,d) = \frac{x-d}{\log (x-d)\, \log x} +  \int_{\frac{x}{(\log x)^5}}^{x-d} \frac{dt}{(\log t)^2 \log (t+d)}+  \int_{\frac{x}{(\log x)^5}}^{x-d} \frac{dt}{\log t ( \log (t+d))^2}+  O(\frac{d}{(\log x)^2})+ O(\frac{x}{(\log x)^5}).
\ee
We now define, for fixed integers $m,n\ge 1$, 
\be 
I_{m,n} := \int_{\frac{x}{(\log x)^5}}^{x-d} \frac{dt}{(\log t)^m (\log (t+d))^n},
\ee
and then obtain  by the same integration by parts argument as above 
\be 
I_{m,n} = \frac{x-d}{(\log x)^n\, (\log (x-d))^m } +  mI_{m+1,n} + nI_{m,n+1} + O(\frac{d}{(\log x)^2})+ O(\frac{x}{(\log x)^5}).
\ee
Furthermore we have the trivial estimate  $I_{m,n} \ll  \frac{x}{(\log x)^{m+n}}$. 
Hence 
\be \begin{split} I(x,d) &= \frac{x-d}{\log x\, \log (x-d)} + I_{2,1} +I_{1,2}  +  O(\frac{d}{(\log x)^2})+ O(\frac{x}{(\log x)^5})\\&
=  \frac{x-d}{\log x\, \log (x-d) } + \frac{x-d}{\log x  (\log (x-d))^2 }+ \frac{x-d}{(\log x)^2 \log (x-d) }+  2I_{3,1} +I_{2,2}  + I_{2,2} + 2 I_{1,3}+\\ & \hskip 1in + O(\frac{d}{(\log x)^2})+ O(\frac{x}{(\log x)^5})\\&
= (x-d)\bigg(\frac{1}{\log x\, \log (x-d) } + \frac{1}{(\log x)^2 \log (x-d) }+ \frac{1}{\log x  (\log (x-d))^2 }+ \\& + \frac{2}{\log x  (\log (x-d))^3 }+\frac{2}{(\log x)^2 (\log (x-d))^2 } + \frac{2}{(\log x)^3\, \log (x-d) }+  O(\frac{d}{x(\log x)^2})+ O(\frac{1}{(\log x)^5})\bigg).\end{split}
\ee 
Since for $2\le d\le \frac89 x$, 
\be \log (x-d) = \log x + \log(1- \frac{d}{x}) = \log x +O(\frac{d}{x}) ,\ee
we conclude that 
\be 
\begin{split} I(x,h) &=    (x-d)\left( \frac{1}{(\log x)^2}+\frac{2}{(\log x)^3} +\frac{6}{(\log x)^4}+ O(\frac{d}{x(\log x)^2}) +O(\frac{1}{(\log x)^5}\right)\left(1 + O(\frac{d}{x\log x})\right) \\&
=\frac{x-d}{(\log x)^2}\left( 1  +\frac{2}{\log x}+\frac{6}{(\log x)^2} +O(\frac{1}{(\log x)^3})+ O(\frac{d}{x\log x}) \right).
\end{split} 
\ee
Letting 
\be  H(x,d) := \frac{I(x,d)}{x-d}(\log x)^2  =  1  +\frac{2}{\log x}+\frac{6}{(\log x)^2} +O(\frac{1}{(\log x)^3})+ O(\frac{d}{x\log x}) ,\end{equation}
 Conjecture 1 takes the form that for $2\le d\le \frac89 x$,
\begin{equation} \label{G-formula}  G(x,d)  =\mathfrak{S}(d)\frac{x-d}{(\log x)^2}H(x,d)\left(1 +o(\frac{1}{(\log x)^2})\right) . 
\ee
This will be the formula we use in what follows. 

Before proceeding to the proof, we explain how this formula implies our theorem. The factor $\frac{H(x,d)}{(\log x)^2}$ is essentially constant and can be ignored.  From \eqref{singprod} the singular series $\mathfrak{S}(d)$ increases on the sequence of primiorials and therefore
if $d< p_{k}^{\sharp}$ then $ \mathfrak{S}(d) < \mathfrak{S}(p_{k}^{\sharp})$.  Thus $G(x,d)$ will also grow on the sequence of primorials as long as $d$ is small enough that the linear decreasing factor $(x-d)$ does not overwhelm the increase from the singular series.  From \eqref{eq5} and the Prime Number Theorem we have
\be \label{eq34} \log p_{k}^{\sharp} = \sum_{p\le p_k} \log p \sim p_k, \qquad \text{as} \ \ k\to \infty,  \ee
which relates the primorial to the primes that constitute it. 
The critical range is when $ x^{1-o(1)}\le p_{k}^{\sharp} \le x$ and thus here $ p_k \sim \log x$ .   The maximum of $G(x,d)$ will occur at the primorial where 
\be 
\mathfrak{S}(p_{k}^{\sharp})(x- p_{k}^{\sharp}) > \mathfrak{S}(p_{k+1}^{\sharp})(x- p_{k+1}^{\sharp})
\ee
for the first time, and this occurs when
\be 
\frac{\mathfrak{S}(p_{k+1}^{\sharp})}{\mathfrak{S}(p_{k}^{\sharp})} \sim \frac{x- p_{k}^{\sharp}}{x- p_{k+1}^{\sharp}}.
\ee
The left-hand side is
\be = 1+ \frac{1}{p_{k+1}-2} \sim 1+\frac{1}{\log x},\ee
while the right-hand side is
\be = 1 + \frac{p_{k+1}^{\sharp}-p_{k}^{\sharp}}{x-p_{k}^{\sharp}} \sim 1 + \frac{p_{k+1}^{\sharp}}{x}.\ee

Hence these two expressions match  when $ p_{k+1}^{\sharp} \sim \frac{x}{\log x}$ and we expect that the PDC will usually be $ \lfloor \frac{x}{\log x}\rfloor^\sharp$, where $\lfloor x\rfloor^\sharp$ is a floor function with respect to the primorials, defined by
\be 
\lfloor x\rfloor^\sharp = p_{k}^{\sharp} \quad \text{if} \quad  p_{k}^{\sharp}\le x < p_{k+1}^{\sharp}. 
\label{4.5}
\ee

The exception and most delicate situation is where the primorial $p_{k+1}^{\sharp}$ is very close to $\frac{x}{\log x}$, in which case either it or the previous primorial $p_{k}^{\sharp}$ may be the PDC. Here $p_{k}^{\sharp}$ will be very close to $\frac{x}{(\log x)^2}$. Notice the singular series has the same value at $p_{k}^{\sharp}$ and $2p_{k}^{\sharp}$, so in order to show this latter value is not the PDC we need to use the inequality $x- p_{k}^{\sharp} > x-2p_{k}^{\sharp}$ in Conjecture 1. These terms differ in Conjecture 1 in the  $\frac{x}{(\log x)^4}$ term, which we can only distinguish by taking $E(x,d) = o(\frac{x}{(\log x)^4})$ for  the error in Conjecture 1. 

The result we prove is the following.
\begin{theorem} Assume Conjecture 1.  Let $0<\delta \le \frac14$ be a given number. Then for $x$ sufficiently large, if the interval 
$[(1+\frac{\delta}{2}) \frac{x}{(\log x)^2}, (1-\frac{\delta}{2}) \frac{x}{\log x}]$ contains a primorial then that primorial is the PDC.  If this interval does not contain a primorial, then both of the intervals 
$[(1-\delta) \frac{x}{(\log x)^2}, (1+\frac{\delta}{2}) \frac{x}{(\log x)^2})$ and  $((1-\frac{\delta}{2}) \frac{x}{\log x}, (1+\delta) \frac{x}{\log x}]$ will contain primorials and one or the other or sometimes both will be the PDCs.
\end{theorem}

\begin{corollary} Assume Conjecture 1. Then all sufficiently large PDC's are primorials, and every sufficiently large primorial will be the PDC for some $x$. \end{corollary}

Before proceeding to the proof of Theorem 1 it is instructive to refer again to the numerical calculations presented in Section \ref{S_Num}.  
The dot-dashed line in Figure \ref{Fig_PDC} is a logarithmic plot of $x/\log(x)$ and represents an asymptotic upper bound on the PDC's in accordance with Theorem 1.  The dashed line is a logarithmic plot of $x/\log^2(x)$ and represents the corresponding lower bound.  The analytical bounds are in excellent agreement with the numerical data.  
Furthermore, in accordance with Corollary 1, we find that the PDC's in the observed range are all primorials for $x\ge19$.  
Note also that the upper bound on the PDC's is essentially the prime counting function, which explains the observed tendency for the steps in the PDC's to occur where $x=p_n$ is roughly equal to the next largest primorial.   

\section{Proof of Theorem 1 \label{S_ProofOfTheorem1}} 
Theorem 1 follows from the following two lemmas.   
First it is convenient to define a ceiling function with respect to the sequence of primorials, analogously to 
$\lfloor x \rfloor^{\sharp}$, such that  
\be
\label{ceiling} 
\lceil x\rceil^\sharp = p_{k}^{\sharp} \quad \text{if} \quad p_{k-1}^{\sharp}< x\le  p_{k}^{\sharp}. 
\ee

\begin{lemma} Assume Conjecture 1, and let  $0<\delta \le \frac14$ be a given number. Define
\be  p_{a}^{\sharp} := \left\lfloor (1 - \frac{\delta}{2} ) \frac{x}{\log x}\right\rfloor^\sharp .
\ee
Then, for $x$ sufficiently large  and $2\le d < p_{a}^{\sharp}$,  we have $G(x,d)<G(x,p_{a}^{\sharp})$.  
\end{lemma}

\begin{lemma} Assume Conjecture 1, and let  $0<\delta \le \frac14$ be a given number. 
Define
\be  p_{b}^{\sharp} := \left\lceil (1 + \frac{\delta}{2} ) \frac{x}{(\log x)^2}\right\rceil^\sharp .
\ee
Then, for $x$ sufficiently large  and $x\ge d>  p_{b}^{\sharp}$,  we have $G(x,p_{b}^{\sharp})>G(x,d)$.  \end{lemma}

\begin{proof}[Proof of Theorem 1]  There can be at most one primorial in the interval $[(1+\frac{\delta}{2}) \frac{x}{(\log x)^2}, (1-\frac{\delta}{2}) \frac{x}{\log x}]$ since if $p_{k}^{\sharp} = x^{1+o(1)}$ we have $p_{k+1}^{\sharp} \sim p_{k}^{\sharp}\log x$  and $\log x$ times the left endpoint of this interval is larger than the right endpoint. If there is a primorial in this interval then clearly it is $p_{a}^{\sharp}= p_{b}^{\sharp}$, and by Lemmas 1 and 2 it is the PDC.  

Next  both of the intervals 
$[(1-\delta) \frac{x}{(\log x)^2}, (1-\frac{\delta}{2}) \frac{x}{\log x})$ and  $((1+\frac{\delta}{2}) \frac{x}{(\log x)^2}, (1+\delta) \frac{x}{\log x}]$ will contain at least one primorial, but if  $[(1+\frac{\delta}{2}) \frac{x}{(\log x)^2}, (1-\frac{\delta}{2}) \frac{x}{\log x}]$ does not contain a primorial then both of the intervals 
$[(1-\delta) \frac{x}{(\log x)^2}, (1+\frac{\delta}{2}) \frac{x}{(\log x)^2})$ and  $((1-\frac{\delta}{2}) \frac{x}{\log x}, (1+\delta) \frac{x}{\log x}]$ must contain consecutive primorials,  say $p_{j}^{\sharp}$ and $p_{j+1}^{\sharp}$, and $p_{j}^{\sharp} = p_{a}^{\sharp}$ and $p_{j+1}^{\sharp}=p_{b}^{\sharp}$. At least one of these must be the PDC, and since, for fixed $d$,  $N(x,d)$ increases by steps of 1 as $x$ increases, the PDC will transition from  $p_{a}^{\sharp}$ to  $p_{b}^{\sharp}$ and there must be at least one value of $x$ where $N(x,p_{a}^{\sharp} ) = N(x,p_{b}^{\sharp})$ and the PDC is $\{ p_{a}^{\sharp},p_{b}^{\sharp}\}$.

\end{proof}

In proving Lemma 1 and Lemma 2 we use \eqref{G-formula} to examine the ratio
\be 
\label{ratio}  \frac{ G(x,p_{k}^{\sharp})}{ G(x,d)} =\frac{ \mathfrak{S}(p_{k}^{\sharp})}{ \mathfrak{S}(d)} \frac{x-p_{k}^{\sharp}}{x-d} \frac{ H(x,p_{k}^{\sharp})}{ H(x,d)}\left(1+o(\frac{1}{(\log x)^2})\right). 
\ee
For $2\le d<  p_{k}^{\sharp}$ we have $\mathfrak{S}( p_{k-1}^{\sharp}) \ge \mathfrak{S}( d)$ and therefore
\be 
\label{singratio} \frac{ \mathfrak{S}(p_{k}^{\sharp})}{ \mathfrak{S}(d)}= \left(1 + \frac{1}{p_k-2}\right)\frac{ \mathfrak{S}(p_{k-1}^{\sharp})}{ \mathfrak{S}(d)}\ge 1 + \frac{1}{p_k-2} . 
\ee
Also,  for $2\le d_1,d_2\le\frac89  x$ 
\be
\label{linearratio}\begin{split} \frac{x-d_2}{x-d_1}&= \frac{1-\frac{d_2}{x}}{1-\frac{d_1}{x}} = \left({1-\frac{d_2}{x}}\right) \left(1+\frac{d_1}{x} +\left( \frac{d_1}{x}\right)^2 +\left(\frac{d_1}{x}\right)^3 + \cdots  \right) \\& \ge
 \left({1-\frac{d_2}{x}}\right) \left(1+\frac{d_1}{x} \right) \\&
= 1+\frac{d_1-d_2}{x} -\frac{d_1d_2}{x^2},\end{split}
\ee
and 
\be
\label{Hratio} \frac{ H(x,d_2)}{ H(x,d_1)} =1+O(\frac{d_1+d_2}{x\log x})+O(\frac{1}{(\log x)^3}).
\ee

\begin{proof}[Proof of Lemma 1]
In Lemma 1  we have $2\le d < p_{a}^{\sharp}\le (1-\frac{\delta}{2})\frac{x}{\log x}$ and  $\mathfrak{S}( p_{a-1}^{\sharp}) \ge \mathfrak{S}( d)$.  Therefore by \eqref{singratio} we have
\be 
\frac{ \mathfrak{S}(p_{a}^{\sharp})}{ \mathfrak{S}(d)}\ge 1 + \frac{1}{p_a-2} =1 +\frac{1}{\log x}(1+o(1));
\ee
by \eqref{linearratio} we have 
\be  
\frac{x-p_{a}^{\sharp}}{x-d}= 1-\frac{p_{a}^{\sharp}}{x}+\frac{d}{x}+O(\frac{1}{(\log x)^2}) \ge 1 - \frac{1-\frac{\delta}{2}}{\log x}  +O(\frac{1}{(\log x)^2}),
\ee
and by \eqref{Hratio} we have
\be 
\frac{ H(x,p_{a}^{\sharp})}{ H(x,d)} =1+O(\frac{1}{(\log x)^2}).
\ee
Therefore by \eqref{ratio} we have
\be 
\begin{split}  \frac{ G(x,p_{a}^{\sharp})}{ G(x,d)} &\ge\left( 1 +\frac{1}{\log x}(1+o(1))\right) \left(  1 - \frac{1-\frac{\delta}{2}}{\log x} +O(\frac{1}{(\log x)^2})\right)\left(1+O(\frac{1}{(\log x)^2})\right)\left(1+o(\frac{1}{(\log x)^2})\right) \\ &
\ge 1 +\frac{ \frac{\delta}{2}}{\log x} -o(\frac{1}{\log x}) \\ &
>  1 +\frac{\frac{\delta}{4}}{\log x} \\ & >1,
\end{split} 
\ee
for all sufficiently large $x$.  This proves Lemma 1. 
\end{proof}

\begin{proof}[Proof of Lemma 2] We now assume $d >  p_{b}^{\sharp}$.  Here $ (1 + \frac{\delta}{2} ) \frac{x}{(\log x)^2}\le p_{b}^{\sharp} \le   (1 + \delta ) \frac{x}{\log x}$ for $x$ sufficiently large. We need to divide the proof of Lemma 2 into cases depending on the size of $d$.  

\medskip
\emph{Case 1.} Suppose $ p_{b}^{\sharp}< d<2 p_{b}^{\sharp}$. In this range 
$ \mathfrak{S}( p_{b-1}^{\sharp}) \ge \mathfrak{S}(d) $ so that  just as before in \eqref{singratio} we have 
\be  
\frac{ \mathfrak{S}(p_{b}^{\sharp})}{ \mathfrak{S}(d)}\ge 1 + \frac{1}{p_b-2} =1 +\frac{1}{\log x}(1+o(1)).
\ee
By \eqref{linearratio} 
\be
\frac{x-p_{b}^{\sharp}}{x-d}= 1-\frac{p_{b}^{\sharp}}{x}+\frac{d}{x}+O\left(\frac{(p_{b}^{\sharp})^2}{x^2}\right) \ge 1  +O(\frac{1}{(\log x)^2}),
\ee
and by \eqref{Hratio}
\be 
\frac{ H(x,p_{b}^{\sharp})}{ H(x,d)} =1+O(\frac{p_{b}^{\sharp}}{x\log x})+  O(\frac{1}{(\log x)^3})=1+O(\frac{1}{(\log x)^2}).
\ee
Hence by \eqref{ratio}
\be 
\begin{split} \frac{ G(x,p_{b}^{\sharp})}{ G(x,d)}& \ge \left(1 +\frac{1}{\log x}(1+o(1))\right)\left(1+O(\frac{1}{(\log x)^2})\right)^2\left(1+o(\frac{1}{(\log x)^2})\right)\\&
\ge 1 + \frac{\frac12}{\log x} \\& >1 \end{split}
\ee
for all sufficiently large $x$, which proves Lemma 2 in this range.

\medskip
\emph{Case 2.}  Suppose $2 p_{b}^{\sharp}\le d< \min( p_{b+1}^{\sharp}, \frac89 x)$. Since $d< p_{b+1}^{\sharp}$,   we have $ \mathfrak{S}( p_{b}^{\sharp}) \ge \mathfrak{S}(d) $, and thus
\be
\frac{ \mathfrak{S}(p_{b}^{\sharp})}{ \mathfrak{S}(d)}\ge 1 .
\ee
Now $d\le \frac89 x$ is needed to stay in the range where \eqref{G-formula}, \eqref{linearratio}, and \eqref{Hratio} are valid. 
Using the inequality  $d- p_{b}^{\sharp} \ge \frac{d}{2}$ which is valid when $2 p_{b}^{\sharp}\le d$, we have by \eqref{linearratio} for $d\le \frac89 x$, 
\be 
\begin{split}  \frac{x-p_{b}^{\sharp}}{x-d}&\ge  1  +\frac{d-p_{b}^{\sharp}}{x}  -\frac{dp_{b}^{\sharp}}{x^2} \\& \ge 1 + \frac{d}{2x}  -\frac{d(1+\delta)}{x\log x} \\& > 1 + \frac{d}{3x} ,\end{split}
\ee
and by \eqref{Hratio} for $d\le \frac89 x$, 
\be 
\frac{ H(x,p_{b}^{\sharp})}{ H(x,d)} =1+O(\frac{d}{x\log x})+  O(\frac{1}{(\log x)^3}) = 1+O(\frac{d}{x\log x}).
\ee
Hence by \eqref{ratio}
\be 
\begin{split} \frac{ G(x,p_{b}^{\sharp})}{ G(x,d)}& \ge \left(1 + \frac{d}{3x} \right)\left(1+O(\frac{d}{x\log x})\right)\left( 1+o(\frac{1}{(\log x)^2})\right) \\&
> 1 +\frac{d}{4x}\\& >1 
\end{split}
\ee
for all sufficiently large $x$, which proves Lemma 2 in this range. Notice we needed $E(x,d) =o(\frac{x}{(\log x)^4})$ in Conjecture 1 for this last step.

\medskip
\emph{Case 3.}  Suppose $  p_{b+1}^{\sharp} \le d< \frac89 x$.  We see that $ p_{b+2 }^{\sharp}\sim (\log x)  p_{b+1}^{\sharp}\sim (\log x)^2  p_{b}^{\sharp} > (1+\frac{\delta}{3}) x $ for $x$ sufficiently large, and hence  $ p_{b+2 }^{\sharp}$ is larger than the range of $d$ here. Hence in this range  $  \mathfrak{S}(p_{b+1}^{\sharp}) \ge  \mathfrak{S}(d)$. Further we have shown  $p_{b+1}^{\sharp} > \frac{x}{\log x}$ and therefore $\frac{x}{\log x} \le d< \frac89 x$.
Hence
\be 
\begin{split} \frac{ \mathfrak{S}(p_{b}^{\sharp})}{ \mathfrak{S}(d)} &=  \frac{ \mathfrak{S}(p_{b+1}^{\sharp})}{ \mathfrak{S}(d)}\left(1+\frac{1}{ p_{b+1}-2}\right)^{-1} \\&
\ge \left( 1 + \frac{1}{\log x}(1+o(1))\right)^{-1} \\&
\ge  1 - \frac{1}{\log x}(1+o(1)).
\end{split}
\ee
By \eqref{linearratio}, and using $p_{b}^{\sharp} \ll \frac{d}{\log x}$ in this range, 
\be 
\begin{split}  \frac{x-p_{b}^{\sharp}}{x-d}&\ge  1  +\frac{d-p_{b}^{\sharp}}{x}  -\frac{dp_{b}^{\sharp}}{x^2} \\& \ge 1 +\frac{d}{x}  - \left(\frac{d}{x}+1\right)\frac{p_{b}^{\sharp}}{x}\\&
\ge 1 +\frac{d}{x}  - \frac{2p_{b}^{\sharp}}{x}\\&
\ge 1 +\frac{d}{x}  + O( \frac{d}{x\log x})\\&
\ge 1 +\frac{d}{x} (1+o(1))
 .\end{split}
 \ee
By \eqref{Hratio} for $\frac{x}{\log x} \le d< \frac89 x$, 
\be 
\frac{ H(x,p_{b}^{\sharp})}{ H(x,d)} =1+O(\frac{d}{x\log x})+ O(\frac{1}{(\log x)^3}) = 1 + o(\frac{d}{x}),
\ee
and hence by \eqref{ratio}
\be 
\begin{split} \frac{ G(x,p_{b}^{\sharp})}{ G(x,d)}& > \left(  1 - \frac{1}{\log x}(1+o(1)) \right)\left( 1 +\frac{d}{x} (1+o(1))\right)\left(1 + o(\frac{d}{x})\right)\left(1+o(\frac{1}{(\log x)^2})\right)\\&
= 1 +\left(\frac{d}{x} - \frac{1}{\log x}\right)(1+o(1))  
\end{split}
\ee
Since here  
\be
\begin{split}\frac{d}{x} &\ge \frac{ p_{b+1}^{\sharp}}{x} = \frac{ p_{b}^{\sharp}\log x(1+o(1))}{x}\\&\ge \frac{ (1+\frac{\delta}{2})\frac{x}{(\log x)^2} \log x (1+o(1))}{x} \\& >\frac{ 1 +\frac{\delta}{3}}{\log x}\end{split}
\ee
for all sufficiently large $x$, we conclude
\be 
\frac{ G(x,p_{b}^{\sharp})}{ G(x,d)} > 1 +\frac{ \frac{\delta}{4}}{\log x} >1
\ee
for all sufficiently large $x$, which proves Lemma 2 in this range. 

\medskip
\emph{Case 4.}  Suppose $ \frac89 x\le d\le x-\frac{x}{(\log x)^3} $.   We use the sieve upper bound, for $1<y\le x$ 
\be 
G(x,d) -G(x-y,d) \le  8 \mathfrak{S}(d) \frac{y}{(\log y)^2}\left( 1 +O(\frac{\log\log3 y}{\log y}\right),
\ee
see Theorem 5.3 or Corollary 5.8.2 of \cite{HalberstamRichert1974}.
Since trivially $G(d,d)=0$, we see that $G(x,d) = G(x,d) - G(d,d) =G(x,d) - G(x-(x-d),d)$, and therefore taking $y =x-d\le \frac{x}{9}$ in the sieve estimate above for $ \frac89 x\le d\le x-\frac{x}{(\log x)^3}$, we have
\be 
\begin{split} G(x,d) & \le  8  \mathfrak{S}(d) \frac{x-d}{(\log (x-d))^2}\left( 1 +O(\frac{\log\log3 (x-d)}{\log (x-d)}\right)   \\& \le  \frac89 \mathfrak{S}(p_{b+1}^{\sharp})\frac{x}{(\log x)^2}(1+o(1)) \\&
< \frac9{10}\mathfrak{S}(p_{b}^{\sharp})\frac{x}{(\log x)^2}.
\end{split}
\ee
On the other hand, by Conjecture 1 in the simple form in \eqref{HLpairs}
\be 
G(x, p_{b}^{\sharp}) \sim \mathfrak{S}(p_{b}^{\sharp})\frac{x}{(\log x)^2}, 
\ee
and thus $ G(x, p_{b}^{\sharp}) > G(x,d)$.

\medskip
\emph{Case 5.} Suppose $  x-\frac{x}{(\log x)^3} \le d \le x $.  Then if $p'-p =d$ then $p' \ge p  + x-\frac{x}{(\log x)^3} > 
 x-\frac{x}{(\log x)^3}$.
Then 
\be
G(x,d)  =  \sum_{\substack{p,p' \leq x \\ p' - p = d}} 1 \ \le \sum_{  x-\frac{x}{(\log x)^3}<p'\le x}1 \ \le \frac{x}{(\log x)^3}
\ee
and Lemma 2 follows since this is smaller than the size of $G(x, p_{b}^{\sharp})$ we just used in Case 4. 

This completes the proof of Lemma 2. 

\end{proof}

\section{ Logarithmically Weighted Sums and Products of Primes \label{S_LogSums}}
Starting with \eqref{singprod}, we take logarithms and obtain
\be\label{5.2}
\mathfrak{S}(d)= 2C_2\exp\left(\sum_{\substack{p \mid d \\ p > 2}}\log\left(1+ \frac{1}{p - 2} \right)\right).
\ee
These types of products and sums over primes with \emph{logarithmic weighting} are much easier to evaluate than  the unweighted primes in the prime number theorem and can be estimated using some elementary results of Merten, see \cite{Ingham1932} or \cite{MontgomeryVaughan2007}.

\begin{lemma}[Merten] We have
\begin{equation}\label{5.3} \prod_{p\le y}\left(1-\frac1{p}\right)^{-1}= e^{\gamma} \log y +O(1), \end{equation}
and
\begin{equation}\label{5.4} \sum_{p\le y} \frac1{p} = \log\log y + b +O(\frac{1}{\log y}),\end{equation}
where $\gamma$ is Euler's constant and $b$ is a constant.
\end{lemma}
We often will make use of the estimate, for fixed constants $0<a<b$, 
\begin{equation} \label{5.5}\sum_{ a\log x \le p \le b\log x} \frac1{p} \ll \frac1{\log\log x} ,\end{equation}
which follows immediately on differencing in \eqref{5.4}. This estimate in turn implies that
\begin{equation}\label{5.6}\begin{split} \prod_{ a\log x \le p \le b\log x} \left(1 + O\left(\frac{1}{p}\right)\right)&= \exp\left( \sum_{ a\log x \le p \le b\log x}\log\left( 1+O\left(\frac1p\right)\right) \right)\\ &= \exp\left( \sum_{ a\log x \le p \le b\log x}O\left(\frac1p\right)\right)\\ &=\exp\left( O(\frac{1}{\log\log x})\right)\\ & = 1+O(\frac1{\log\log x}). \end{split} \end{equation}
  
Now consider the divisor sum
\begin{equation}\label{5.7} M(d) := \sum_{p|d}\frac1{p}, \end{equation}
and let
\begin{equation}\label{5.8} M^*(x) := \max_{1\le d\le x} M(d).\end{equation}
It is clear that 
\begin{equation}
\label{5.9} M^*(x) = M(\lfloor x\rfloor^\sharp) = \sum_{p\le p_k}\frac1{p}
\end{equation}
where $\lfloor x\rfloor^\sharp = p_{k}^{\sharp}$. We have by \eqref{eq34} that  $p_k \sim \log x$ and therefore by Lemma 3
\begin{equation}
\label{5.10}M^*(x) = \log\log\log x +O(1)  . 
\end{equation}
Finally, returning to \eqref{5.2}, we have by \eqref{5.10}
\begin{equation}
\label{5.11} \mathfrak{S}(d) \ll \exp\left( M(d) +O(1)\right)\  \le \ \exp\left( M^*(d) +O(1)\right) \ll \log\log d. 
\end{equation}

\section{The Prime Difference Champions Go to Infinity \label{S_PDCsInfinity}}
We now prove the PDC's go to infinity and have many prime factors.
\begin{theorem} Suppose $d^*=d^*(x)$ is a prime difference champion for primes $\le x$. Then $d^* \to \infty$ as $x\to \infty$ and the number of distinct prime factors of $d^*$ also goes to infinity as $x\to \infty$.
\end{theorem}

First, we need a simplier version of the same sieve bound for prime pairs we used in the proof of Case 4 of Lemma 2.

\begin{lemma}  We have for $d$ positive and even
\begin{equation}  G(x,d) \le \mathcal{C}\, \mathfrak{S}(d) \frac{x}{(\log x)^2}\left(1+o(1)\right), \end{equation}
uniformly for $1\le d\le x$, where $\mathcal{C}$ is a constant.
\end{lemma}
The value $\mathcal{C}=4$ follows from the Bombieri-Vinogradov theorem, see \cite{HalberstamRichert1974}. (Slightly smaller values of $\mathcal{C}$ are known to hold.) 

Our next lemma finds values of $d$ for which $G(x,d)$ is large and thus gives a lower bound on $G^*(x)$. 

\begin{lemma} Let $1\le q\le x $.  Then
\begin{equation}\sum_{1\le m\le \frac{x}{q}} G(x,mq) \ge \frac12 \left(\frac{x^2}{\phi(q)(\log x)^2}- \frac{x}{\log x}\right)(1+o(1)). \end{equation}
\end{lemma} 
\begin{proof} Let 
$\pi(x;q,a)$ denote the number of primes $\le x$ which are congruent to $a$ modulo $q$. 
Since
\be 
\sum_{\substack{1\le a\le q\\(a,q)=1}}\left(
\pi(x;q,a)- \frac{\pi(x)}{\phi(q)}\right)^2 \ge 0,
\ee
we have on multiplying out that
\begin{equation} 
\sum_{\substack{1\le a\le q\\(a,q)=1}}\pi(x;q,a)^2\ge 2 \frac{\pi(x)}{\phi(q)}\sum_{\substack{1\le a\le q\\(a,q)=1}}\pi(x;q,a)-\frac{\pi(x)^2}{\phi(q)}.
\end{equation}
When $(a,q)>1$ we have $\pi(x;q,a)= 0 \ \text{or} \ 1$  Therefore
\be 
\sum_{\substack{1\le a\le q\\(a,q)>1}}\pi(x;q,a)=\sum_{\substack{1\le a\le q\\(a,q)>1}}\pi(x;q,a)^2 \le \sum_{p|q}1 \ll \log q,
\ee
and therefore we  drop the condition $(a,q)=1$ in both sums above and obtain
\begin{equation} \sum_{1\le a\le q}\pi(x;q,a)^2\ge 2 \frac{\pi(x)}{\phi(q)} \sum_{1\le a\le q}\pi(x;q,a)-\frac{\pi(x)^2}{\phi(q)} +O(\frac{\pi(x)\log q}{\phi(q)})+O(\log q).
\end{equation}
Now
\be 
\sum_{1\le a\le q}\pi(x;q,a) = \pi(x),
\ee
and
\be
\begin{split}  \sum_{1\le a\le q}\pi(x;q,a)^2 &=  \sum_{1\le a\le q}\Big(\sum_{\substack{p,p'\le x\\ p'\equiv p \equiv a \, (\text{mod }q)}}1\Big) \\ & = \sum_{\substack{p,p'\le x\\ p'\equiv p  (\text{mod }q)}}1 \\&
= \sum_{\substack{-x\le d\le x\\ q|d}} \sum_{\substack{p,p'\le x\\ p'-p=d}}1\\&
= \pi(x) + 2\sum_{\substack{1\le d\le x\\ q|d}}G(x,d),
\end{split}
\ee
and therefore we obtain
\begin{equation} \label{6.5} \begin{split} 2\sum_{\substack{1\le d\le x\\ q|d}}G(x,d) &\ge \frac{\pi(x)^2}{\phi(q)} -\pi(x) +O(\frac{\pi(x)\log q}{\phi(q)})+O(\log q)\\& \ge \left(\frac{x^2}{\phi(q)(\log x)^2}- \frac{x}{\log x}\right)(1+o(1)),\end{split} \end{equation}
where in the last line we used the prime number theorem in the form
\begin{equation} \label{6.6} \pi(x) = \frac{x}{\log x}(1+o(1)). \end{equation}  
\end{proof} 

\begin{proof}[Proof of Theorem 2] Since 
\be 
G^*(x) \frac{x}{q} \ge \sum_{1\le m\le \frac{x}{q}} G(x,mq),
\ee
by Lemma 5 we conclude that, recalling $\phi(q) = q\prod_{p|q}(1- \frac{1}{p})$,  
\begin{equation} \label{6.7} \begin{split} G^*(x)&\ge \frac12 \frac{q}{\phi(q)} \left(\frac{x}{(\log x)^2}- \frac{q}{\log x}\right)(1+o(1)) \\ & = \frac12 \prod_{p|q}\left(1-\frac1{p}\right)^{-1} \left(\frac{x}{(\log x)^2}- \frac{q}{\log x}\right)(1+o(1)) . \end{split} \end{equation}
We now choose $ q=\lfloor \frac{x}{(\log x)^2}\rfloor^\sharp= p_{\ell}^{\sharp}$, and thus  $p_\ell \sim \log x$. Hence 
\be 
\label{6.8}  G^*(x) \ge \frac12\prod_{p\le p_\ell}\left(1-\frac1{p}\right)^{-1}\frac{x}{(\log x)^2}(1  + o(1)).
\ee
By Lemma 4, we have that if $d^*$ is a PDC for primes $\le x$, then 
\begin{equation} \label{6.9}  G^*(x) \le \mathcal{C}\, \mathfrak{S}(d^*) \frac{x}{(\log x)^2}\left(1+o(1)\right).\end{equation}
Combining \eqref{6.8} and \eqref{6.9} we have
\begin{equation} \label{6.10}  \mathcal{C}\, \mathfrak{S}(d^*)  \ge \frac12\prod_{p\le p_{\ell}}\left(1-\frac1{p}\right)^{-1}(1  + o(1)), \end{equation}
which by Lemma 3 gives
\be  
\mathfrak{S}(d^*)  \ge \frac{e^\gamma}{2\mathcal{C}}\log p_{\ell}(1  + o(1))\gg \log\log x, 
\ee
and hence by the first part of \eqref{5.11} we obtain
\begin{equation} \label{6.11}
\sum_{p \mid d^*} \frac1p \gg \log\log\log x .
\end{equation}

Therefore $d^*\to \infty$ and the number of distinct prime factors of $d^*$  also go to infinity, as $ x \to \infty$, which proves Theorem 2. 
\end{proof}
We can obtain a slightly more precise result with a little more effort.
\begin{theorem} We have
\begin{equation} \label{6.12}   \sum_{\substack{p\le 2\log x\\ p  \nmid d^*(x)}}\frac1{p} \le \log(2\mathcal{C})(1+o(1))< 2.08. \end{equation}
\end{theorem}

\begin{corollary} We have
\be 
\sum_{p \mid d^*(x)} \frac1p \sim M^*(x) \sim \log\log\log x.
\ee
\end{corollary}
Hence we see that the PDCs must have asymptotically the maximal number of prime factors when the factors are weighted logarithmically. However this does not imply that the PDCs must have many small prime factors, but we can prove they must have a few \lq \lq small" prime factors.

\begin{corollary} For sufficiently large $x$, every PDC is divisible by an odd prime $\le 25,583$.
\end{corollary}

\begin{proof}[Proof of Theorem 3]
Continuing from \eqref{6.10} and using \eqref{3.2} we obtain
\begin{equation} \label{6.14}  2C_2\prod_{\substack{p \mid d^* \\ p > 2}} \left(\frac{p - 1}{p - 2}\right)\prod_{\substack{p\le p_\ell\\ p>2}}\left(1-\frac1{p}\right)\ge \frac1{\mathcal{C}}(1  + o(1)). \end{equation}
We write the left-hand side of this equation as
\begin{equation} \label{6.15}  2C_2\prod_{\substack{p \mid d^* \\ p > p_\ell}} \left(\frac{p - 1}{p - 2}\right)\prod_{\substack{p \mid d^* \\ 2<p\leq p_\ell}} \left(\frac{p - 1}{p - 2}\right)\left(1-\frac1{p}\right)\prod_{\substack{p\le p_\ell\\ p  \nmid d^*}}\left(1-\frac1{p}\right) = 2C_2 P_1P_2P_3. \end{equation}

First, just as in \eqref{5.5},
\be
\begin{split} 1\le P_1 &\le \prod_{\substack{p \mid d^* \\ p > p_\ell}}\left(1+\frac{1}{p - 2}\right)\\&
\le \prod_{\substack{p \mid \lfloor x \rfloor^\sharp \\ p > p_\ell}}\left(1+\frac{1}{p - 2}\right)\\&
\le \prod_{\substack{\frac12 \log x \le p \le 2\log x }}\left(1+\frac{1}{p - 2}\right)\\&
= 1+O(\frac{1}{\log\log x}) .\end{split}
\ee
Thus $P_1 = 1+o(1)$. 
Next, since 
\be 
\left(\frac{p - 1}{p - 2}\right)\left(1-\frac1{p}\right) = \frac{(p-1)^2}{p(p-2)} = \left( 1 - \frac1{(p-1)^2}\right)^{-1}, 
\ee
we see
\be 
P_2 \le \prod_{p>2}\left( 1 - \frac1{(p-1)^2}\right)^{-1} = \frac1{C_2}.
\ee
Finally, by \eqref{5.6} 
\be 
P_3 = \prod_{\substack{p\le 2\log x\\ p  \nmid d^*}}\left(1-\frac1{p}\right)\left(1+O(\frac{1}{\log\log x})\right).
\ee
We conclude from \eqref{6.14} and \eqref{6.15} that 
\begin{equation}\label{6.16} \prod_{\substack{p\le 2\log x\\ p  \nmid d^*}}\left(1-\frac1{p}\right) \ge  \frac1{2\mathcal{C}}(1  + o(1)). \end{equation}
Taking logarithms we have
\be  
\sum_{\substack{p\le 2\log x\\ p  \nmid d^*}}-\log \left(1-\frac1{p}\right) \le  \log(2\mathcal{C})(1+o(1)). 
\ee
Since  
\be 
- \log(1-x) =  x+\frac{x^2}{2} +\frac{x^3}{3} +\cdots  \qquad \text{for}\  |x|<1,
\ee
we see  $ -\log(1- x) > x$ for $0<x<1$
and hence taking $\mathcal{C}=4$,
\be  
\sum_{\substack{p\le 2\log x\\ p  \nmid d^*}}\frac1{p} \le \log(8)(1+o(1))< 2.08, 
\ee
which proves \eqref{6.12}.
\end{proof}
Corollary 2 follows from \eqref{5.10} and \eqref{6.12}.
Corollary 3 follows from \eqref{6.12} because 
\be
\sum_{3\le p \le  p_{2817}}\frac1p > \log 8
\ee
 and $p_{2817} = 25583$.
Hence every sufficiently large PDC is divided by both $2$ and an odd prime $\le 25583$.

\end{document}